\documentclass[10pt]{article}

\usepackage{mathptmx}
\usepackage{longtable}
\usepackage{array}
\usepackage{amsmath}
\usepackage{amscd}
\usepackage{amssymb}
\usepackage{amsthm}
\usepackage{xspace}
\usepackage[all,tips]{xy}
\usepackage[dvips]{graphicx}
\usepackage{graphics,color}
\usepackage{epsfig}
\usepackage{verbatim}
\usepackage{syntonly}
\usepackage{hyperref,float}

\newtheorem{theorem}{Theorem}[section]

\newtheorem{corollary}[theorem]{Corollary}
\newtheorem{lemma}[theorem]{Lemma}
\newtheorem{proposition}[theorem]{Proposition}

\input xy
\xyoption{all}

\DeclareMathOperator{\GL}{GL}\DeclareMathOperator{\PSL}{PSL}
\DeclareMathOperator{\SL}{SL}

\DeclareMathOperator{\D}{D}

\DeclareMathOperator{\DM}{F}

\DeclareMathOperator{\M}{S}
\DeclareMathOperator{\FL}{L}
\DeclareMathOperator{\PSU}{PSU}
\DeclareMathOperator{\PSp}{PSp}\DeclareMathOperator{\PGL}{PGL}

\DeclareMathOperator{\Alt}{Alt}

\DeclareMathOperator{\Z}{\mathbf{Z}}

\def\Z{\mathbf{Z}}
\def\N{\mathbf{N}}
\def\F{\mathbf{F}}

\def\rw{\rightarrow}

\newcommand{\innp}[1]{\left< #1 \right>}
\newcommand{\abs}[1]{\left\vert#1\right\vert}
\newcommand{\norm}[1]{\left\vert\left\vert #1\right\vert\right\vert}
\newcommand{\set}[1]{\left\{#1\right\}}

\newcommand{\su}{\subset}

\usepackage{fancyhdr}

\fancypagestyle{plain}

\begin{document}
\bibliographystyle{plain}


\title{\textbf{Characterizing linear groups \\ in terms of growth properties}}
\author{Khalid Bou-Rabee\thanks{The City College of New York, New York, NY. E-mail: \tt{khalid.math@gmail.com}}~~and D.~B.~McReynolds\thanks{Purdue University, West Lafayette, IN. E-mail: \tt{dmcreyno@purdue.edu}}}
\maketitle


\begin{abstract}
\noindent 
Residual finiteness growth measures how well-approximated a group is by its finite quotients. We prove that some related growth functions characterize linearity for a class of groups including all hyperbolic groups.
\end{abstract}

\section{Introduction}

Given a finitely generated, residually finite group $\Gamma$ and a non-trivial $\gamma \in\Gamma$, we consider three functions that measure the difficulty of verifying that $\gamma$ is non-trivial through homomorphisms to finite groups. The first function $\D_\Gamma(\gamma)$ is the minimum $\abs{Q}$, over all finite groups $Q$, such that there exists a homomorphism $\varphi\colon \Gamma \to Q$ with $\varphi(\gamma) \ne 1$. This function was introduced in \cite{B10}. In this paper, we consider two variations of $\D_\Gamma$ obtained by restricting the class of groups $Q$. We define $\FL_\Gamma(\gamma)$ to be the minimum $\abs{\GL(n,\F_q)}$, over all $n \in \N$ and finite fields $\F_q$, such that there exists a homomorphism $\varphi\colon \Gamma \to \GL(n,\F_q)$ with $\varphi(\gamma) \ne 1$. We define $\M_\Gamma(\gamma)$ to be the minimal $\abs{G}$, over all finite simple groups $G$, such that there exists a homomorphism $\varphi\colon \Gamma \to G$ with $\varphi(\gamma) \ne 1$. Note that we do not insist that any of the above homomorphisms be surjective, though for $\D_\Gamma$, the minimum $\abs{Q}$ for an element $\gamma$ always comes from a surjective homomorphism. Fixing a finite generating subset $X$ for $\Gamma$, for each $m \in \N$, we define
\[ \DM_\Gamma(m) = \max_{\substack{\gamma \in \Gamma - \set{1}, \\ \norm{\gamma}_X\leq m}} \D_\Gamma(\gamma), \quad \DM_{\Gamma,\FL}(m) = \max_{\substack{\gamma \in \Gamma - \set{1}, \\ \norm{\gamma}_X\leq m}} \FL_\Gamma(\gamma), \quad \DM_{\Gamma,\M}(m) = \max_{\substack{\gamma \in \Gamma - \set{1}, \\ \norm{\gamma}_X\leq m}} \M_\Gamma(\gamma). \]
For two functions $f,g\colon \N \rw \N$, we write $f \preceq g$ if there exists a natural number $C$ such that $f(m) \leq Cg(Cm)$ holds for all $m$ and write $f \sim g$ when $f \preceq g$ and $g \preceq f$. The dependence of the above functions on the generating subset $X$ is mild. Specifically, if $Y$ is another generating subset, the associated functions for $X,Y$ satisfy the equivalence relation $\sim$ (see \cite[Lem 1.1]{B10}). In \cite{BM15}, we proved that if $\Gamma \leq \GL(n,\mathbf{K})$ for some $n \in \N$ and field $\mathbf{K}$, then $\DM_\Gamma(m) \preceq m^D$ for $D \in \N$ depending only on $n$ and $\mathbf{K}$. The main result of this article is the following characterization of linearity (property (c)) for hyperbolic groups $\Gamma$  in terms of the growth rates of  $\DM_{\Gamma,\FL},\DM_{\Gamma,\M}$.

\begin{theorem}\label{T:Main}
Let $\Gamma$ be a hyperbolic group. The following are equivalent:
\begin{itemize}
\item[(a)]
There exists $D \in \N$ such that $\DM_{\Gamma,\FL}(m) \preceq m^D$.
\item[(b)]
There exists $D' \in \N$ such that $\DM_{\Gamma,\M}(m) \preceq m^{D'}$.
\item[(c)]
There exists $n \in \N$, a field $\mathbf{K}$, and an injective homomorphism $\varphi\colon \Gamma \to \GL(n,\mathbf{K})$.
\end{itemize}
\end{theorem}

The first purely group theoretic characterization of (c), for fields $\mathbf{K}$ with $\mathrm{char}(\mathbf{K})=0$, was established by Lubotzky \cite{Lub88}, and is based on the existence of uniformly powerful $p$--filtrations.  Larsen \cite{L01} proved that the abscissa of convergence of a certain subgroup $\zeta$--function is the inverse of the minimal possible dimension of the Zariski closure over all finite dimensional representations with infinite image. Though the representation that Larsen produces need not be faithful, his method is closely related to our approach which is based on a well-known, elementary characterization of (c); see Lemma \ref{L:UltraLemma} below. 


\paragraph{Acknowledgements}
The authors are grateful to Noel Brady, Mikhail Ershov, Benson Farb, Daniel Groves, Richard Kent, Michael Larsen, Lars Louder, Denis Osin, Alan Reid, Ignat Sokoro, and Henry Wilton for conversations on this topic. The first author was partially supported by NSF RTG Grant DMS-0602191, NSF Grant DMS-1405609, and by Ventotene 2013. The second author was partially supported by NSF grants DMS-1105710 and DMS-1408458. Finally, a huge debt is owed to the referee for numerous helpful suggestions and corrections to an earlier draft of the paper.

\section{Finite Groups}

In this section, we record some results about finite groups that we will need for Theorem \ref{T:Main}. Lemma \ref{L:LarsenLemma0} is the main technical result.

For a finite group $G$, set $m_1(G) = \max_{g \in G} \abs{\innp{g}}$. For each prime $p \in \N$, set $r_p(G)$ to be the minimal positive integer $n$ such that there exists $t \in \N$ and an injective homomorphism $\varphi\colon G \to \PGL(n,\F_{p^t})$. We define $r(G) = \min_{p \text{ prime}} r_p(G)$. We define $r_p^{\FL}(G)$, $r^{\FL}(G)$ identically but with $\GL_n$ in place of $\PGL_n$ and note that if $G$ is simple, we have $r(G) \leq r^{\FL}(G)$.
Further, since $\PGL_n(K) \leq \GL_{n^2}(K)$ for any field $K$, we have $r^{\FL}(G) \leq (r(G))^2$ for any group $G$.

Throughout, $\mathcal{F} = \set{G_j}_{j \in \N}$ will denote a set of finite groups. For our purposes, we will assume that the groups $G_j$ are pairwise non-isomorphic. We say that $\mathcal{F}$ has \textbf{bounded projective rank for some $R \in \N$} if $r(G_j) \leq R$ for all $j \in \N$ and has \textbf{bounded rank for some $R \in \N$} if $r^{\FL}(G_j) \leq R$ for all $j$. For any non-principal ultrafilter $\omega$ on $\N$, $G_\omega$ will denote the ultraproduct group $\prod_\omega G_j$ of $\mathcal{F}$ relative to $\omega$. The following is well-known (see \cite{Hal95} for instance).

\begin{lemma}\label{L:UltraLemma}
If $\mathcal{F} = \set{G_j}_{j \in \N}$ is a set  of finite groups with either bounded rank or bounded projective rank for some $R \in \N$, then for any (non-principal) ultrafilter $\omega$ on $\N$, there is an injective homomorphism $\varphi_\omega\colon G_\omega \to \GL(n,\mathbf{K})$ for some $n \in \N$ and field $\mathbf{K}$.
\end{lemma}

The integer $n$ can be taken to be $R^2$ in both cases since $r^{\FL}(G) \leq (r(G))^2$.

\subsection{A criterion for bounded rank}\label{SS:FG}

The following lemma was suggested to us by the referee.

\begin{lemma}\label{L:LarsenLemma0}
Let $\mathcal{F}=\set{G_j}_{j \in \N}$ be a set of finite groups.
\begin{itemize}
\item[(a)] If $G_j = \GL(n_j,\F_{q_j})$ for all $j$, then $\frac{\log \abs{G_j}}{\log(m_1(G_j))}$ is unbounded if and only if $r^{\FL}(G_j)$ is unbounded.
\item[(b)]
If $G_j$ is simple for all $j$, then $\frac{\log \abs{G_j}}{\log(m_1(G_j))}$ is unbounded if and only if $r(G_j)$ is unbounded.
\end{itemize}
\end{lemma}

To prove this lemma we have compiled the requisite information for certain families of finite groups. For a finite simple group of Lie type $G(\F_q)$ over $\F_q$ with $q=p^t$, we refer the reader \cite{Che55}, \cite{Car72}, \cite{KL90}, or \cite{Wil12}. The orders of these groups are well-known and the bounds we provide are coarse; we will use these bounds in the sequel without specific mention. The monograph \cite[Ch 5]{KL90} contains a complete summary on bounds for $r(G)$ for both alternating groups and simple groups of Lie type; see Proposition 5.3.7, Table 5.4.C, and Corollary 5.4.14 in \cite{KL90}. Lower bounds for $r_{p'}(G(\F_q))$ for primes $p' \ne p$ have been extensively studied; see \cite{TZ96} and the references therein. The representation theory over fields with the defining characteristic $p$ has also been extensively studied. In most cases, $r_p(G(\F_q)) = r(G(\F_q))$. The values for $r(G)$ given below are from Table 5.4.C and Corollary 5.4.14 in \cite{KL90}. We write $f \approx g$ when $\lim_q \frac{f(q)}{g(q)} = C \in (0,\infty)$. For computational ease, with regard to the proof of Lemma \ref{L:LarsenLemma0}, we have opted for simple, coarse inequalities over precise values/asymptotics in most situations. That is especially true for the values $m_1$ and $\abs{G}$. Finally, our use of the classification of finite simple groups is at least that Lemma \ref{L:LarsenLemma0} hold for the sporadic groups. That is weaker than the assertion that there are only finitely many sporadic groups.

\paragraph{(1) $\mathrm{Alt}(n)$ and $\GL(n,\F_q)$.}

For the alternating groups $\Alt(n)$ for $n \geq 5$, we have:
\begin{align*}
\abs{\mathrm{Alt}(n)} &= n!/2, \\ 
\log(m_1(\mathrm{Alt}(n))) &\approx \sqrt{n\log n}, \\ 
r(\mathrm{Alt}(n)) &\leq n-2, \quad r(\mathrm{Alt}(n)) = n-2 \text{ for }n \geq 9.
\end{align*}
The first estimate for $m_1$ is due to Landau \cite{L03} (see (1) in \cite{Mil87} for the above explicit asymptotics). We see that
\begin{equation}\label{E:Alt}
C_1\frac{\log(n!/2)}{\sqrt{n\log n}} \leq \frac{\log\abs{\mathrm{Alt}(n)}}{\log(m_1(\mathrm{Alt}(n)))} < C_2 \frac{\log(n!/2)}{\sqrt{n\log n}},
\end{equation}
for fixed constants $C_1,C_2>0$. For the groups $\GL(n,\F_q)$ for $n \geq 1$ and $q=p^t$, we have
\begin{align*}
q^{n^2-1} &< \abs{\GL(n,\F_q)} < q^{n^2}, \\ 
q^{n-1} &< m_1(\GL(n,\F_q)) < q^{n+1}, \\ 
n-2 &\leq r^{\FL}(\GL(n,\F_q)) \leq n.
\end{align*}
By Niven \cite{Niv48} (see also \cite[Cor 2]{Dar05}), $m_1(\GL(n,\F_q))=q^n-1$. The lower bound for $r^{\FL}$ is obtained by using derived/solvable lengths of solvable subgroups of $\GL(n,\F_q)$ (see Corollary on p. 152 of \cite{Dix68}). We see that
\begin{equation}\label{E:GL}
n-1 \leq \frac{\log\abs{\GL(n,\F_q)}}{\log(m_1(\GL(n,\F_q)))} < \frac{n^2}{n-1}.
\end{equation}

\paragraph{(2) Type $\mathrm{A}_n$.}

The groups of type $\mathrm{A}_n(q)$ for $n \geq 1$ and $q=p^t$ are $\PSL(n+1,q)$. We have
\begin{align*} 
q^{n^2-3} &< \abs{\PSL(n+1,q)} < q^{n^2 +2n +1}, \\ 
q^{n-2} &< m_1(\PSL(n+1,q)) < q^{n+2}, \\
r(\PSL(n+1,q)) &= r_p(\PSL(n+1,q))  = n+1, \quad (\text{for }(n,q) \ne (1,4), (1,5), (2,2)).
\end{align*}
The numbers $m_1$ are explicitly calculated in \cite[Table A.1]{KS}. We see that
\begin{equation}\label{E:An}
\frac{(n^2-3)}{(n+2)} < \frac{\log \abs{\PSL(n+1,q)}}{\log(m_1(\PSL(n+1,q)))} < \frac{n^2 +2n +1}{(n-2)}.
\end{equation}

\paragraph{(3) Type $^2 \mathrm{A}_n$.}

The groups of type $^2\mathrm{A}_n(q^2)$ for $n\geq 2$ and $q=p^t$ are $\PSU(n+1,q)$. We have
\begin{align*}
q^{n^2+2n-1} &< \abs{\PSU(n+1,q)} < q^{n^2 +2n +1}, \\ 
q^{n-4} &< m_1(\PSU(n+1,q)) < q^{n+1},  \\
r(\PSU(n+1,q)) &= r_p(\PSU(n+1,q)) = n+1, \quad (\text{for }(n,q) \ne (3,2)).
\end{align*}
The numbers $m_1$ are explicitly calculated in \cite[Table A.2]{KS}. We see that
\begin{equation}\label{E:An2}
\frac{n^2+2n-1}{n+1} < \frac{\log\abs{\PSU(n+1,q)}}{\log(m_1(\PSU(n+1,q)))} < \frac{n^2+2n+1}{n-4}.
\end{equation}

\paragraph{(4) Type $\mathrm{B}_n$.}

The groups of type $\mathrm{B}_n(q)$ for $n\geq 2$ and $q=p^t$ are given by $\mathrm{P}\Omega(2n+1, q)$. We have
\begin{align*}
q^{2n^2+n-2} &< \abs{\mathrm{P}\Omega(2n+1, q)} < q^{2n^2+n+1}, \\ 
q^{n-2} &< m_1(\mathrm{P}\Omega(2n+1, q)) < q^{n+2}, \\ 
r(\mathrm{P}\Omega(2n+1, q)) &= r_p(\mathrm{P}\Omega(2n+1, q)) = 2n+1, \quad (\text{for odd }q \text{ and }n \ne 2), \\
r(\mathrm{P}\Omega(2n+1, q)) &= r_p(\mathrm{P}\Omega(2n+1, q)) = 2n, \quad (\text{for } n=2 \text{ or even }q). 
\end{align*}

The numbers $m_1$ are explicitly calculated in \cite[Table A.4]{KS} for odd $q$ and \cite[Cor 4]{But} when $q=2^t$ (see also \cite{Spiga} since $\mathrm{B}_n(2^t) = \mathrm{C}_n(2^t)$). We see that
\begin{equation}\label{E:Bn}
\frac{2n^2+n-2}{n+2} < \frac{\log\abs{\mathrm{P}\Omega(2n+1, q)}}{\log(m_1(\mathrm{P}\Omega(2n+1, q)))} < \frac{2n^2+n+1}{n-2}.
\end{equation}

\paragraph{(5) Type $\mathrm{C}_n$.}

The groups of type $\mathrm{C}_n(q)$ for $n\geq 3$ and $q=p^t$ are given by $\PSp(2n,q)$. We have
\begin{align*}
q^{2n^2+n-2} &< \abs{\PSp(2n,q)} < q^{2n^2+n+1}, \\ 
q^{n-2} &< m_1(\PSp(2n,q)) < q^{n+2}, \\ 
r(\PSp(2n,q)) &= r_p(\PSp(2n,q) = 2n.
\end{align*}
The numbers $m_1$ are explicitly calculated in \cite[Table A.3]{KS} for $q$ odd and \cite[Thm 1.1]{Spiga} (see also Lemma 2.3 in \cite{Spiga}) when $q$ is $2^t$. We see that
\begin{equation}\label{E:Cn}
\frac{2n^2+n-2}{n+2} < \frac{\log\abs{\PSp(2n,q)}}{\log(m_1(\PSp(2n,q)))} < \frac{2n^2+n+1}{n-2}.
\end{equation}

\paragraph{(6) Type $\mathrm{D}_n$.}

The groups of type $\mathrm{D}_n(q)$ for $n\geq 4$ and $q=p^t$ are given by $\mathrm{P}\Omega^+(2n, q)$. We have
\begin{align*}
q^{2n^2-n-2} &< \abs{\mathrm{P}\Omega^+(2n, q)} < q^{2n^2 -n+1}, \\ 
q^{n-2} &< m_1(\mathrm{P}\Omega^+(2n, q)) < q^{n+2}, \\ 
r(\mathrm{P}\Omega^+(2n,q)) &= r_p(\mathrm{P}\Omega^+(2n,q)) = 2n, \quad (\text{for } (n,q) \ne (4,2)
\end{align*}
The numbers $m_1$ are explicitly calculated in \cite[Table A.5]{KS} for odd $q$ and \cite[Cor 4]{But} when $q=2^t$. We see that
\begin{equation}\label{E:Dn}
\frac{2n^2-n-2}{n+2} < \frac{\log\abs{\mathrm{P}\Omega^+(2n, q)}}{\log(m_1(\mathrm{P}\Omega^+(2n, q)))} < \frac{2n^2-n+1}{n-2}.
\end{equation}

\paragraph{(7) Type $^2 \mathrm{D}_n$.}

The groups of type $^2\mathrm{D}_n(q^2)$ for $n\geq 4$ and $q=p^t$ are given by $\mathrm{P}\Omega^-(2n, q)$. We have
\begin{align*}
q^{2n^2-n-2} &< \abs{\mathrm{P}\Omega^-(2n, q)} < q^{2n^2 -n+1}, \\ 
q^{n-2} &< m_1(\mathrm{P}\Omega^-(2n, q)) < q^{n+2}, \\ 
r(\mathrm{P}\Omega^-(2n, q)) &= r_p(\mathrm{P}\Omega^-(2n, q)) = 2n.
\end{align*}
The numbers $m_1$ are explicitly calculated in \cite[Table A.6]{KS} for odd $q$ and \cite[Cor 4]{But} when $q=2^t$. We see that
\begin{equation}\label{E:Dn2}
\frac{2n^2-n-2}{n+2} < \frac{\log\abs{\mathrm{P}\Omega^-(2n, q)}}{\log(m_1(\mathrm{P}\Omega^-(2n, q)))} < \frac{2n^2-n+1}{n-2}.
\end{equation}

Excluding alternating groups, in all the above simple families $\frac{\log\abs{G(n,q)}}{\log(m_1(G(n,q)))} \approx r(G(n,q))$ (in the parameter $n$).

\paragraph{(8) Exceptional Groups Lie type.}

For the remaining exceptional finite simple groups of Lie type and cyclic groups of prime order, both $r(G)$ and $\frac{\log\abs{G}}{\log(m_1(G))}$ are uniformly bounded above by $248$; for the sporadic groups, $196883$ is an upper bound for both $r(G)$ and $\frac{\log\abs{G}}{\log(m_1(G))}$. We have included a table summarizing the information for the exceptional and cyclic families. We refer the reader to Table 5.4.C and Corollary 5.4.14. in \cite{KL90} for the values of $r$ and to Table A.7 in \cite{KS} for the values of $m_1$. 

 \begin{table}[H]
\begin{center}
\newcolumntype{Q}{ >{$}c <{$}}
\begin{tabular}{|Q | Q | Q | Q |}
\hline
\mbox{Family}   &  \mbox{$\abs{G} \approx$} & \mbox {$m_1(G) \approx $} & \mbox{$r(G)$} \\
\hline
 \Z/p\Z, \, p \text{ prime}   &  p &  p  & 1 \\
 \textrm{E}_6 (q)    &   q ^ {78} & q^6  &  27 \\
 \textrm{E}_7 (q)   &   q ^ {133} & q^7 & 56 \\
 \textrm{E}_8 (q)   &   q ^ {248} & q^8  & 248 \\
 \textrm{F}_4 (q)   &   q ^ {52}  & q^4  & 25 \leq r(G) \leq 26\\
 \textrm{G}_2 (q)   &   q ^ {14} &  q^2  & 6 \leq	 r(G) \leq 7 \\
 ^2\textrm{E}_6 (q^ 2)    &  q^{78}  & q^6  & 27 \\
 ^3\textrm{D}_4 (q ^ 3)   &  q ^ {28} & q^4   & 8 \\
 ^2\textrm{B}_2 (2 ^ {2 j + 1})    &  q^5,  \mbox{where } q = 2 ^ {2 j + 1} & q^2  & 4 \\
 ^2\textrm{G}_2 (3 ^ {2 j + 1})    &  q^2 , \mbox{where } q = 3 ^ {2 j + 1} & q^2 & 7 \text{ (for }j \ne 1) \\
 ^2\textrm{F}_4 (2 ^ {2 j + 1})    &   \  q^{26}, \mbox{where } q = 2 ^ {2 j + 1}  & q^2  & 26 \text{ (for }j \ne 1) \\
\hline
\end {tabular}
\caption{}
\label{table:FS3}
\end{center}
\end{table}

\begin{proof}[Proof of Lemma \ref{L:LarsenLemma0}]
In the case $G_j = \GL(n_j,\F_{q_j})$, we see from \eqref{E:GL} that $\frac{\log \abs{G_j}}{\log(m_1(G_j))}$ is unbounded if and only if $r^{\FL}(G_j)$ is unbounded. In the case $G_j$ is simple for all $j$, if either $r(G_j)$ or $\frac{\log\abs{G_j}}{\log(m_1(G_j))}$ is unbounded, we know that there is a subsequence $\set{G_{j_i}}$ of $\set{G_j}$ where none of the terms are sporadic or in Table \ref{table:FS3}. Passing to another subsequence, we can assume that for all $i \in \N$, either $G_{j_i} = \mathrm{Alt}(n_{j_i})$ or $G_{j_i}$ is of Lie type and in precisely one of the families $\mathrm{A}_{n_{j_i}}(q_{j_i}),~^2\mathrm{A}_{n_{j_i}}(q_{j_i}^2),\mathrm{B}_{n_{j_i}}(q_{j_i}),\mathrm{C}_{n_{j_i}}(q_{j_i}),\mathrm{D}_{n_{j_i}}(q_{j_i}),~^2\mathrm{D}_{n_{j_i}}(q_{j_i}^2)$. In the alternating case, we see from \eqref{E:Alt} and the simple inequality, $n! \geq (n/e)^n$, that $\frac{\log \abs{G_{j_i}}}{m_1(G_{j_i})}$ is unbounded if and only if $r(G_{j_i})$ is unbounded. We see that the same holds for each of those families of Lie type by \eqref{E:An}, \eqref{E:An2}, \eqref{E:Bn}, \eqref{E:Cn}, \eqref{E:Dn}, \eqref{E:Dn2}, and the listed $r$--values.
\end{proof}

The following immediate consequence of Lemma \ref{L:LarsenLemma0} will be used later.

\begin{lemma}\label{L:LarsenLemma}
For every $C>0$, there exists $R(C) \in \N$ such that if $G = \GL(n,\F_q)$ for some $n \in \N$ and finite field $\F_q$, and $\frac{\log \abs{G}}{\log(m_1(G))} \leq C$, then $r^{\FL}(G) \leq R$. Similarly, for every $C>0$, there exists $R(C) \in \N$ such that if $G$ is a finite simple group and $\frac{\log \abs{G}}{\log(m_1(G))} \leq C$, then $r(G) \leq R$.
\end{lemma}

We require one additional result on finite simple groups.

\begin{lemma}\label{L:PowerControl}
If $\mathcal{F} = \set{G_j}_{j \in \N}$ is set of finite simple groups that has bounded projective rank for some $R \in \N$, then there is a $D(R) \in \N$ such that for each $j \in \N$, there is a finite field $\F_{q_j}$ and an injective homomorphism $\psi_{G_j}\colon G_j \to \PSL(R,\F_{q_j})$ with $\abs{\PSL(R,\F_{q_j})} \leq \abs{G_j}^D$.
\end{lemma}

\begin{proof}
From \eqref{E:Alt}, we see that only finitely many alternating groups can be in $\mathcal{F}$ and set $\mathcal{F}_0$ to be the finite subset of sporadic and alternating groups in $\mathcal{F}$. There exists $D_1(\mathcal{F}_0) \in \N$ such that for each $G \in \mathcal{F}_0$, there exists a finite field $\F_q$ and an injective homomorphism $\psi_G\colon G \to \PSL(R,\F_q)$ such that $\abs{\PSL(R,\F_q)} \leq \abs{G}^{D_1}$. For any $G \in \mathcal{F} - \mathcal{F}_0$, either $G$ is cyclic of prime order or of Lie type. If $G$ is cyclic of prime order $p$, we have an injective homomorphism $\psi_G\colon G \to \PSL(2,\F_p)$ and so $\abs{\PSL(2,\F_p)} < \abs{G}^4$. If $G$ is of Lie type with associated finite field $\F_q$, since $r(G)\leq R$, we see from the data in Subsection \ref{SS:FG} that there exists $D_2(R) \in \N$ and an injective homomorphism $\psi_G\colon G \to \PSL(R,\F_q)$ such that $\abs{\PSL(R,\F_q)} \leq \abs{G}^{D_2}$; $D_2 = R^2$ works. Finally, set $D = \max\set{D_1,D_2,4}$.
\end{proof}

\section{Infinite Groups}

In this section, we establish the main technical result Proposition \ref{P:MainTechProp-GL}. From Proposition \ref{P:MainTechProp-GL}, we obtain Corollary \ref{C:Main} which we will use in the proof of Theorem \ref{T:Main}. We also introduce malabelian and uniformly malabelian groups.

Given a group $\Gamma$ and a set $\mathcal{F} = \set{G_j}_{j\in \N}$ of finite groups, we say that $\Gamma$ is \textbf{residually $\mathcal{F}$} if for each non-trivial $\gamma \in \Gamma$, there exists $G \in \mathcal{F}$ and a homomorphism $\varphi\colon \Gamma \to G$ such that $\varphi(\gamma) \ne 1$. We say that $\Gamma$ is \textbf{fully residually $\mathcal{F}$} if for each finite set $T \su \Gamma - \set{1}$, there exists $G \in \mathcal{F}$ and a homomorphism $\varphi\colon \Gamma \to G$ such that $\ker \varphi \cap T = \emptyset$. If $\mathcal{F} = \set{\GL(n,\F_q)}$ or $\set{\PGL(n,\F_q)}$ for $n \in \N$ fixed and varying finite fields $\F_q$, we will simply say $\Gamma$ is (fully) residually $\GL_n$ or $\PGL_n$; these families have bounded rank, bounded projective rank, respectively. The following result is also well-known (see Theorem 3.11 in \cite{Hal95} for instance).

\begin{lemma}\label{L:UltraLemmaApp}
Let $\Gamma$ be a countable group and $\mathcal{F} = \set{G_j}_{j \in \N}$ be a set of finite groups that has bounded rank or bounded projective rank. If $\Gamma$ is fully residually $\mathcal{F}$, then there exists $n \in \N$, a field $\mathbf{K}$, and an injective homomorphism $\varphi\colon \Gamma \to \GL(n,\mathbf{K})$.
\end{lemma}

\begin{proof}
We begin by enumerating $\Gamma = \set{\gamma_0,\gamma_1,\dots,}$ with $\gamma_0 = 1$ and set $T_i = \set{\gamma_k}_{k=1}^i$. By assumption, for each $i \in \N$, there is a $j_i \in \N$ and a homomorphism $\varphi_{j_i}\colon \Gamma \to G_{j_i}$ such that $\varphi_{j_i}(\gamma) \ne 1$ for all $\gamma \in T_i$. Setting $\mathcal{F}_0 = \set{G_{j_i}}_{i \in \N}$, for any non-principal ultrafilter $\omega$ on $\N$, we have an injective homomorphism $\varphi_\omega\colon \Gamma \to \prod_\omega G_{j_i}$. The proof is completed with an application of Lemma \ref{L:UltraLemma}.
\end{proof}

\subsection{Malabelian Groups}

Given a finite subset $T \su \Gamma - \set{1}$, we have an associated normal subgroup $\Delta_T = \bigcap_{\gamma \in T} \overline{\innp{\gamma}}$, where $\overline{\innp{\gamma}}$ denotes the normal closure of the cyclic subgroup $\innp{\gamma}$. We call any non-trivial element in $\Delta_T$ a \textbf{common multiple} for $T$ in $\Gamma$. The following lemma can be found in \cite[Lem 3.1]{BM11}.

\begin{lemma}\label{L:CMProp}
Let $\Gamma$ be a group, $T \su \Gamma - \set{1}$ a finite subset, and $\eta$ a common multiple for $T$ in $\Gamma$. If $G$ is a group and $\varphi\colon \Gamma \to G$ is a homomorphism with $\varphi(\eta) \ne 1$, then $\varphi(\gamma) \ne 1$ for all $\gamma \in T$.
\end{lemma}

We say that a group $\Gamma$ is \textbf{malabelian} if for any pair of (not necessarily distinct) non-trivial elements $\gamma,\eta \in \Gamma$, there exists $\mu \in \Gamma$ such that $[\mu^{-1}\gamma\mu,\eta] \ne 1$. In other words, every non-trivial conjugacy class in $\Gamma$ has a trivial centralizer.

\begin{lemma}\label{L:CM-Exist}
If $\Gamma$ is a malabelian group and $T \su \Gamma - \set{1}$ a finite subset, then $\Delta_T \ne 1$ and $T$ has a common multiple.
\end{lemma}

\begin{proof}
Without loss of generality, we can assume that $\abs{T} = 2^j$ for some integer $j$ as we can enlarge $T$ by adding non-trivial elements from $\Gamma$. We proceed via induction on $j$. Enumerating $T = \set{\gamma_1,\dots,\gamma_{2^j}}$, for each $1 \leq i \leq 2^{j-1}$, we select $\mu_i \in \Gamma$ such that $\gamma_{i,1} = [\mu_i^{-1}\gamma_{2i-1}\mu_i,\gamma_{2i}] \ne 1$. We obtain a new subset $T_1 = \set{\gamma_{i,1}}_{i=1}^{2^{j-1}}$. By induction, $T_1$ has a common multiple in $\Gamma$ and it is follows from our selection of $\gamma_{i,1}$ that this common multiple is also a common multiple for $T$.
\end{proof} 

\begin{lemma}\label{L:ResGL-UniResGL}
If $\Gamma$ is a countable, malabelian group that is residually $\mathcal{F}$ for some set $\mathcal{F}$ of finite groups, then $\Gamma$ is fully residually $\mathcal{F}$.
\end{lemma}

\begin{proof}
Given a finite subset $T \su \Gamma - \set{1}$, since $\Gamma$ is malabelian, by Lemma \ref{L:CM-Exist} there exists a common multiple $\gamma_T$ for $T$ in $\Gamma$. By assumption, there exists $G \in \mathcal{F}$ and a homomorphism $\varphi_T\colon \Gamma \to G$ such that $\varphi_T(\gamma_T) \ne 1$. By Lemma \ref{L:CMProp}, $\varphi_T(\gamma) \ne 1$ for all $\gamma \in T$ and thus $\Gamma$ is fully residually $\mathcal{F}$.
\end{proof}

\subsection{Uniformly Malabelian Groups}

For a finitely generated group $\Gamma$, a fixed finite generating subset $X$, and $K \in \N$, we say that $\Gamma$ is \textbf{$K$--malabelian with respect to $X$} if for every pair of non-trivial $\gamma,\eta \in \Gamma$, there exists $\mu \in \Gamma$ with $\norm{\mu}_X \leq K$ such that $[\mu^{-1}\gamma\mu,\eta] \ne 1$. If $\Gamma$ is $K$--malabelian with respect to a finite generating subset $X$ and $Y$ is another finite generating subset for $\Gamma$, then $\Gamma$ is $K'$--malabelian with respect to $Y$ for some $K' \in \N$. We say that $\Gamma$ is \textbf{uniformly malabelian} if $\Gamma$ is $K$--malabelian with respect to $X$ for some finite generating subset $X$ and $K \in \N$. For a finite subset $T \su \Gamma - \set{1}$, if $T$ has a common multiple in $\Gamma$, we define the \textbf{least common multiple length} of $T$ relative to $X$ to be $\mathrm{LCM}_X(T) = \min_{\eta \in \Delta_T - \set{1}} \norm{\eta}_X$. We refer to any element $\eta \in \Delta_T$ with $\norm{\eta}_X = \mathrm{LCM}_X(T)$ as a \textbf{least common multiple} for the subset $T$.

\begin{lemma}\label{L:LCM-Est}
If $\Gamma$ is a finitely generated, uniformly malabelian group and $X$ a finite generating subset, then there exists $C(X) \in \N$ such that if $T \su \Gamma - \set{1}$ with $\abs{T} = t$ and $\norm{\gamma}_X \leq d$ for all $\gamma \in T$, then $\mathrm{LCM}_X(T) \leq Cdt^2$.
\end{lemma}

As the proof of Lemma \ref{L:LCM-Est} is similar to the proof of Proposition 4.1 in \cite{BM11}, we have omitted it. 

\begin{proposition}\label{P:MainTechProp-GL}
Let $\Gamma$ be a finitely generated, uniformly malabelian group with an element of infinite order. 
\begin{itemize}
\item[(a)] 
If there exists $D \in \N$ such that $\DM_{\Gamma,\FL}(m) \preceq m^D$, then there exists $n(D) \in \N$ such that $\Gamma$ is residually $\GL_n$.
\item[(b)]
If there exists $D' \in \N$ such that $\DM_{\Gamma,\M}(m) \preceq m^D$, then there exists $R(D') \in \N$ such that $\Gamma$ is residually $\mathcal{F}$ for a set of finite simple groups that has bounded projective rank for $R$.
\end{itemize}
\end{proposition}

\begin{proof}
Before proving (a) and (b), we require some common setup for both. We first fix an infinite order element $\gamma_0$ and a finite generating subset $X$.  Since $\Gamma$ is uniformly malabelian, there exists $K \in \N$ such that $\Gamma$ is $K$--malabelian with respect to $X$. Given a non-trivial $\gamma \in \Gamma$, there exists $\mu_0 \in \Gamma$ such that $[\mu_0^{-1}\gamma\mu_0,\gamma_0] \ne 1$ and $\norm{\mu_0}_X \leq K$. Set $T_j = \set{[\mu_0^{-1}\gamma\mu_0,\gamma_0],\gamma_0^2,\dots,\gamma_0^j}$ and note that $\abs{T_j} = j$. Moreover, there exists $j(\gamma) \in \N$ such that if $j \geq j(\gamma)$, then $\norm{\tau}_X \leq j\norm{\gamma_0}_X$ for all $\tau \in T_j$. By Lemma \ref{L:LCM-Est}, there exists $C_1(X) \in \N$ such that if $\eta_j$ is a least common multiple of $T_j$ and $j \geq j(\gamma)$, then
\begin{equation}\label{E:LCM-Est}
\norm{\eta_j}_X \leq C_1\norm{\gamma_0}_Xj^3 = C_2j^3,
\end{equation} 
where $C_2 = C_1\norm{\gamma_0}_X$.

Part (a). We must show that there exists $n \in \N$ such that for every non-trivial $\gamma \in \Gamma$, there exists a finite field $\F_q$ and a homomorphism $\varphi\colon \Gamma \to \GL(n,\F_q)$ such that $\varphi(\gamma) \ne 1$.
By assumption, there exists $C_3 \in \N$ such that for each $j \in \N$, there exists $n_j \in \N$, a finite field $\F_{q_j}$, and a homomorphism $\varphi_j\colon \Gamma \to \GL(n_j,\F_{q_j})$ with $\varphi_j(\eta_j) \ne 1$ and 
\begin{equation}\label{E:GL-Est}
\abs{\GL(n_j,\F_{q_j})} \leq C_3\norm{\eta_j}_X^D. 
\end{equation}

Combining \eqref{E:LCM-Est}, \eqref{E:GL-Est}, for $C_4 = C_2^DC_3$ and $j \geq j(\gamma)$, we see that $\abs{\GL(n_j,\F_{q_j})} \leq C_4j^{3D}$. By construction, $m_1(\GL(n_j,\F_{q_j})) \geq j$ and so 
\begin{equation}\label{E:Est1}
\frac{\log\abs{\GL(n_j,F_{q_j})}}{\log(m_1(\GL(n_j,F_{q_j})))} \leq 3D + \frac{\log(C_4)}{\log j}.
\end{equation}
Since the right-hand side of \eqref{E:Est1} is uniformly bounded above, by Lemma \ref{L:LarsenLemma}, there exists $R(D,C_4) \in \N$ such that $r^{\FL}(\GL(n_j,F_{q_j})) \leq R$ for all $j \geq j(\gamma)$. Note that since the right-hand side of \eqref{E:Est1} is independent of $\Gamma$, the constant $R(D,C_4)$ is independent of $\Gamma$. Setting $n=R(D,C_4)$, we see that $\Gamma$ is residually $\GL_n$ since $R(D,C_4)$ is independent of $\gamma$.

Part (b). By assumption, there exists $C_5 \in \N$ such that for each $j \in \N$, there exists a finite simple group $G_j$ and  homomorphism $\varphi_j\colon \Gamma \to G_j$ with $\varphi_j(\eta_j) \ne 1$ and 
\begin{equation}\label{E:Sim-Est}
\abs{G_j} \leq C_5\norm{\eta_j}_X^{D}.
\end{equation}
Combining \eqref{E:LCM-Est}, \eqref{E:Sim-Est}, for $C_6=C_2^{D}C_5$ and $j \geq j(\gamma)$, we have $\abs{G_j} \leq C_6j^{3D}$. As before, $m_1(G_j) \geq j$ and so
\begin{equation}\label{E:Format2}
\frac{\log\abs{G_j}}{\log(m_1(G_j))} \leq 3D + \frac{\log(C_6)}{\log j}.
\end{equation}
Since the right-hand of \eqref{E:Format2} is uniformly bounded above, by Lemma \ref{L:LarsenLemma}, there exists $R(D,C_6) \in \N$ such that $r(G_j) \leq R$ for all $j \geq j(\gamma)$. Note that since the right-hand side of \eqref{E:Format2} is independent of $\Gamma$, the constant $R(D,C_6)$ is independent of $\Gamma$. Consequently,  $\Gamma$ is residually $\mathcal{F}$ for some set of finite simple groups that has bounded projective rank.
\end{proof}

\begin{corollary}\label{C:MainLin}
Let $\Gamma$ be a finitely generated, uniformly malabelian group with an element of infinite order. 
\begin{itemize}
\item[(a)]
If there exists $D \in \N$ such that $\DM_{\Gamma,\FL}(m) \preceq m^D$, then there exists $n(D) \in \N$, a field $\mathbf{K}$, and an injective homomorphism $\varphi\colon \Gamma \to \GL(n,\mathbf{K})$. Moreover, $\varphi(\Gamma) \cap (\mathbf{K}^\times \cdot \mathrm{I}_n) = \mathrm{I}_n$.
\item[(b)]
If there exists $D' \in \N$ such that $\DM_{\Gamma,\M}(m) \preceq m^{D'}$, then there exists $n(D') \in \N$, a field $\mathbf{K}$, and an injective homomorphism $\varphi\colon \Gamma \to \GL(n,\mathbf{K})$. Moreover, $\varphi(\Gamma) \cap (\mathbf{K}^\times \cdot \mathrm{I}_n) = \mathrm{I}_n$.
\end{itemize}
\end{corollary}

\begin{proof}
The first part of (a) and (b) follow from Lemma \ref{L:UltraLemmaApp}, Lemma \ref{L:ResGL-UniResGL}, and Proposition \ref{P:MainTechProp-GL}. In either case, we have an injective homomorphism $\varphi\colon \Gamma \to \GL(n,\mathbf{K})$ for some $n \in \N$ and field $\mathbf{K}$. Moreover, it follows from the construction of $\varphi$ that $\varphi(\gamma) \notin \mathbf{K}^\times \cdot \mathrm{I}_n$ for any non-trivial $\gamma \in \Gamma$ since $\varphi_j(\gamma) \notin Z(\varphi_j(\Gamma))$, where $Z(\varphi_j(\Gamma))$ is the center of $\varphi_j(\Gamma)$.
\end{proof}

The converse of Corollary \ref{C:MainLin} follows from the proof of Theorem 1.1 in \cite{BM15} and does not require that $\Gamma$ be malabelian. As $\GL(n,\mathbf{K}) < \SL(n+1,\mathbf{K})$, we will assume that $\Gamma <\SL(n,\mathbf{K})$ in (b).

\begin{proposition}\label{P:LinCon2}
Let $\Gamma$ be a finitely generated group.
\begin{itemize}
\item[(a)]
If $\Gamma \leq \GL(n,\mathbf{K})$ for some $n \in \N$ and field $\mathbf{K}$, then there exists $D(n,\mathbf{K}) \in \N$ such that $\DM_{\Gamma,\FL}(m) \preceq m^D$.
\item[(b)]
If $\Gamma \leq \SL(n,\mathbf{K})$ for some $n \in \N$ and field $\mathbf{K}$ with $\Gamma \cap (\mathbf{K}^\times \cdot \mathrm{I}_n) = \mathrm{I}_n$, then there exists $D(n,\mathbf{K}) \in \N$ such that $\DM_{\Gamma,\M}(m) \preceq m^D$.
\end{itemize}
\end{proposition}

\begin{proof}
Part (a) follows immediately from the proof Theorem 1.1 of \cite{BM15}. For (b), given a non-trivial element $\gamma \in \Gamma$, by \cite{BM15} there exist $C_0,D_0 \in \N$ (independent of $\gamma$), a finite field $\F_q$ with $q \leq C_0\norm{\gamma}_X^{D_0}$, and a homomorphism $\varphi_q\colon \Gamma \to \SL(n,\F_q)$ with $\varphi_q(\gamma) \ne 1$. Moreover, since $\gamma \notin \mathbf{K}^\times \cdot \mathrm{I}_n$, we can arrange for $\varphi_q(\gamma) \notin \F_q^\times \cdot\mathrm{I}_n$. Composing $\varphi_q$ with the homomorphism $\SL(n,\F_q) \to \PSL(n,\F_q)$, we obtain $\mathrm{P}\varphi_q\colon \Gamma \to \PSL(n,\F_q)$ with $\mathrm{P}\varphi_q(\gamma) \ne 1$. By selection of $q$, we have $\abs{\PSL(n,\F_q)} \leq q^{n^2} \leq C_0^{n^2}\norm{\gamma}_X^{D_0n^2}$. As $\gamma$ was arbitrary, we have $\DM_{\Gamma,\M}(m) \preceq m^{D_0n^2}$. 
\end{proof}

Combining Corollary \ref{C:MainLin} and Proposition \ref{P:LinCon2}, we obtain the following corollary.

\begin{corollary}\label{C:Main}
If $\Gamma$ is a finitely generated, uniformly malabelian group with an element of infinite order, then the following are equivalent:
\begin{itemize}
\item[(a)]
There exists $D \in \N$ such that $\DM_{\Gamma,\FL}(m) \preceq m^D$.
\item[(b)]
There exists $D' \in \N$ such that $\DM_{\Gamma,\M}(m) \preceq m^{D'}$.
\item[(c)]
There exists $n \in \N$, a field $\mathbf{K}$, and an injective homomorphism $\varphi\colon \Gamma \to \GL(n,\mathbf{K})$.
\end{itemize}
\end{corollary}

We note that the equivalence of (a) and (c) is independent of the classification of finite simple groups while the equivalence of (b) and (c) is not.

\section{Proof of Theorem \ref{T:Main}}

In the proof of Theorem \ref{T:Main}, we will need to relate the functions $\DM_{\Gamma,\FL},\DM_{\Gamma_0,\FL}$ and $\DM_{\Gamma,\M},\DM_{\Gamma_0,\M}$ for a finitely generated group $\Gamma$ and a finite index subgroup $\Gamma_0$. 

\begin{proposition}\label{P:Invariance}
Let $\Gamma$ be a finitely generated group and $\Gamma_0 \leq \Gamma$ a finite index subgroup.
\begin{itemize}
\item[(a)]
There exists $D \in \N$ such that $\DM_{\Gamma,\FL}(m) \preceq m^D$ if and only if there exists $D' \in \N$ such that $\DM_{\Gamma_0,\FL}(m) \preceq m^{D'}$.
\item[(b)]
If $\Gamma_0$ is uniformly malabelian with an element of infinite order, then there exists $D \in \N$ such that $\DM_{\Gamma,\M}(m) \preceq m^D$ if and only if there exists $D' \in \N$ such that $\DM_{\Gamma_0,\M}(m) \preceq m^{D'}$.
\end{itemize}
\end{proposition}

\begin{proof} 
The direct implications for (a) and (b) are straightforward. For the reverse implications in (a) and (b), we assume $\Gamma_0$ is normal as we can replace $\Gamma_0$ with a finite index subgroup $\Gamma_1\leq \Gamma_0$ with $\Gamma_1 \lhd \Gamma$. 

For (a), given a non-trivial $\gamma \in \Gamma$, we split into two cases. First, if $\gamma \in \Gamma - \Gamma_0$, then we have $\varphi_0\colon \Gamma \to Q_0 = \Gamma/\Gamma_0$ with $\varphi_0(\gamma) \ne 1$. Fixing $n_0 \in \N$ and a finite field $\F_{q_0}$ with $Q_0 \leq \GL(n_0,\F_{q_0})$, we see that $\FL_\Gamma(\gamma) \leq \abs{\GL(n_0,\F_{q_0})}$. If $\gamma \in \Gamma_0$, by assumption there exists $n \in \N$, a finite field $\F_q$, and a homomorphism $\varphi\colon \Gamma_0 \to \GL(n,\F_q)$ with $\varphi(\gamma) \ne 1$ and $\abs{\GL(n,\F_q)} \leq C_0\norm{\gamma}^{D'}$ for $C_0 \in \N$ that is independent of $\gamma$. Setting $\varphi' = \mathrm{Ind}_{\Gamma_0}^\Gamma(\varphi)$, we have $\varphi'\colon \Gamma \to \GL(n\ell_0,\F_q)$ with $\varphi'(\gamma) \ne 1$, where $\ell_0 = [\Gamma:\Gamma_0]$. Since $\abs{\GL(n\ell_0,\F_q)} \leq q^{n^2\ell_0^2} \leq \abs{\GL(n,\F_q)}^{2\ell_0^2}$, we see that $\FL_\Gamma(\gamma) \leq C\norm{\gamma}^{2\ell_0^2D'}$ for some $C \in \N$ that is independent of $\gamma$.

For (b), if $\gamma \in \Gamma - \Gamma_0$, we argue as in (a). When $\gamma \in \Gamma_0$, additional work is required. By Proposition \ref{P:MainTechProp-GL}, there exists a set of finite simple groups $\mathcal{F}$ such that $\Gamma_0$ is residually $\mathcal{F}$ and $\mathcal{F}$ has bounded projective rank for some $R \in \N$.  By Lemma \ref{L:PowerControl}, there exists $D''(R) \in \N$ such that for each $G \in \mathcal{F}$, there is a finite field $\F_q$ and an injective homomorphism $\psi_G\colon G \to \PSL(R,\F_q)$ with $\abs{\PSL(R,\F_q)} \leq \abs{G}^{D''}$. For $\gamma \in \Gamma_0$, by assumption there exists a finite simple group $G \in \mathcal{F}$ and a homomorphism $\varphi_1\colon \Gamma_0 \to G$ such that $\varphi_1(\gamma) \ne 1$ and $\abs{G} \leq C''\norm{\gamma}^{D'}$ for some $C'' \in \N$ that is independent of $\gamma$. Composing $\varphi_1$ and $\psi_G$, we have $\varphi_2\colon \Gamma_0 \to \PSL(R,\F_q)$ with $\varphi_2(\gamma) \ne 1$. Setting $\varphi = \mathrm{Ind}_{\Gamma_0}^\Gamma(\varphi_2)$, we obtain $\varphi\colon \Gamma \to \PSL(R\ell_0,\F_q)$. As $G$ is simple and $\varphi_1(\gamma) \ne 1$, we see that $\varphi(\gamma) \ne 1$. Since
\[ \abs{\PSL(R\ell_0,\F_q)} \leq q^{R^2\ell_0^2} \leq \abs{\PSL(R,\F_q)}^{R^2\ell_0^2} \leq \abs{G}^{D''R^2\ell_0^2} \leq (C'')^{D''R^2\ell_0^2}\norm{\gamma}^{D'D''R^2\ell_0^2}, \]
we see that $\M_\Gamma(\gamma) \leq C'\norm{\gamma}^D$ for $D= D'D''R^2\ell_0^2$ and $C' = (C'')^{D/D'}$. 
\end{proof}

The malabelian assumption in (b) seems like overkill but gives the needed control on $r(G)$ to use induction. One could instead take a wreath product with $\Gamma/\Gamma_0$ and $G$ but from that we still need to find a reasonably small finite simple group that contains it; the left regular representation is too large.

\begin{proof}[Proof of Theorem \ref{T:Main}]
The theorem is straightforward for elementary hyperbolic groups and so we assume $\Gamma$ is non-elementary. If $\Gamma$ is not residually finite, then by Mal'cev \cite{Mal40}, $\Gamma$ does not satisfy (c). Moreover, both functions $\DM_{\Gamma,\FL}(m),\DM_{\Gamma,\M}(m)$ are infinite for sufficiently large $m$.

If $\Gamma$ is residually finite, then there exists a torsion-free, finite index subgroup $\Gamma_0 \leq \Gamma$ and it is straightforward to see that $\Gamma_0$ is uniformly malabelian (see \cite[p. 459--462]{BH99}). By Corollary \ref{C:Main}, properties (a), (b), and (c) are equivalent for $\Gamma_0$. By Proposition \ref{P:Invariance}, $\Gamma_0$ satisfies (a), (b) if and only $\Gamma$ satisfies (a), (b). That $\Gamma$ satisfies (c) if and only if $\Gamma_0$ satisfies (c) follows from induction/restriction.
\end{proof}



\end{document}